\documentclass[a4paper]{article}

\usepackage[utf8]{inputenc}
\usepackage{lmodern}
\usepackage[T1]{fontenc}
\usepackage{microtype}

\usepackage{enumitem}

\usepackage{mathtools, amssymb, amsthm}
\usepackage{enumitem}
\usepackage{bbm,dsfont}
\numberwithin{equation}{section}

\DeclareMathOperator{\T}{T}

\DeclareMathOperator{\supp}{supp}
\newcommand{\MR}{\textit{MR}}
\newcommand{\1}{\mathds{1}}
\newcommand{\K}{\mathds{K}}
\newcommand{\R}{\mathds{R}}
\newcommand{\C}{\mathds{C}}
\newcommand{\N}{\mathds{N}}
\newcommand{\A}{\mathcal{A}}
\newcommand{\B}{\mathcal{B}}

\renewcommand{\L}{\mathcal{L}}
\renewcommand{\d}{\mathrm{d}}
\renewcommand{\Re}{\operatorname{Re}}

\newcommand{\fra}{\mathfrak{a}}

\renewcommand{\mid}{\, \vert \,}

\DeclarePairedDelimiter\abs{\lvert}{\rvert}
 \DeclarePairedDelimiter\norm{\lVert}{\rVert}

\theoremstyle{plain}
\newtheorem{theorem}{Theorem}[section]
\newtheorem{proposition}[theorem]{Proposition}
\newtheorem{lemma}[theorem]{Lemma}
\newtheorem{corollary}[theorem]{Corollary}

\theoremstyle{definition}

\theoremstyle{remark}

\begin{document}
\allowdisplaybreaks
\title{Non-Autonomous Maximal Regularity for Forms of Bounded Variation}
\author{Dominik Dier}

\maketitle

\begin{abstract}\label{abstract}
We consider a non-autonomous evolutionary problem
\[
u' (t)+\A(t)u(t)=f(t), \quad u(0)=u_0,
\]
where $V, H$ are Hilbert spaces such that $V$ is continuously and densely embedded in $H$ and the operator $\A(t)\colon V\to V^\prime$ is associated with a 
coercive, bounded, symmetric form $\fra(t,.,.)\colon V\times V \to \C$ for all $t \in [0,T]$.
Given $f \in L^2(0,T;H)$, $u_0\in V$ there exists always a unique solution $u \in \MR(V,V'):= L^2(0,T;V) \cap H^1(0,T;V')$.
The purpose of this article is to investigate whether $u \in H^1(0,T;H)$.
This property of \emph{maximal regularity in $H$} is not known in general. We give a positive answer if the form is of \emph{bounded variation};
i.e., if there exists a bounded and non-decreasing function $g \colon [0,T] \to \R$ such that
\begin{equation*}
	\abs{\fra(t,u,v)- \fra(s,u,v)} \le [g(t)-g(s)] \norm{u}_V \norm{v}_V \quad (s,t \in [0,T], s \le t).
\end{equation*}
In that case, we also show that $u(.)$ is continuous with values in $V$.
Moreover we extend this result to certain perturbations of  $\A(t)$.
\end{abstract}

\medskip 


\section{Introduction}
The aim of the present article is to study maximal regularity for evolution equations governed by non-autonomous forms.
More precisely, let $T>0$, let $V, H$ be Hilbert spaces such that $V$ is continuously and densely embedded in $H$ and let
\[
	\fra\colon [0,T] \times V \times V \to \C
\]
be a \emph{non-autonomous form}; i.e., $\fra(t,.,.)$ is sesquilinear for all $t \in [0,T]$ and
$\fra(.,v,w)$ is measurable for all $v,w \in V$.
Moreover we assume that there exists constants $M$ and $\alpha >0$ such that
\begin{equation*}
	\abs{\fra(t,v,w)} \le M \norm{v}_V \norm{w}_V \quad(t \in [0,T], v,w \in V)
\end{equation*}
and 
\begin{equation*}
	\Re \fra(t,v,v) \ge \alpha \norm{v}_V^2 \quad(t \in [0,T], v \in V).
\end{equation*}

Then for $t \in [0,T]$ we may define the \emph{associated operator} $\A(t)\in \L(V,V')$ of $\fra(t,.,.)$ by
\[
	\langle \A(t) v, w \rangle = \fra(t,v,w) \quad (v,w \in V).
\]
Here $V'$ denotes the antidual of $V$ and $\langle ., . \rangle$ denotes the duality between $V'$ and $V$.
Note that $H^1(0,T;V') \hookrightarrow C([0,T];V')$, so we may identify every element of $H^1(0,T;V')$ by its continuous representative.
Now a classical result of Lions (see \cite[p.\ 513]{DL92}, \cite[p. 112]{Sho97}) states the following.
\begin{theorem}\label{thm:Lions}
	For every $f \in L^2(0,T;V')$ and $u_0 \in H$ there exists a unique
	\[
		u \in \MR(V,V'):= L^2(0,T;V) \cap H^1(0,T;V')
	\]
	such that
	\begin{equation}\label{eq:CPinV'}
		\left\{\begin{aligned}
			u' + \A u &= f \quad\text{in }L^2(0,T;V')\\
			u(0)&=u_0.
		\end{aligned} \right.
	\end{equation}
	Moreover $\MR(V,V') \hookrightarrow C([0,T];H)$ and
	\begin{equation}\label{eq:MRV'est}
		\norm{u}_{L^2(0,T;V)}^2 \le \tfrac 1 {\alpha^2} \norm{f}_{L^2(0,T;V')} + \tfrac 1 {\alpha} \norm{u_0}_H^2.
	\end{equation}
\end{theorem}
Let $f \in L^2(0,T;H)$, $u_0 = 0$ and let $u \in \MR(V,V')$ be the solution of \eqref{eq:CPinV'}.
In the autonomous case; i.e., if $\fra(t,.,.) = \fra(0,.,.)$ for all $t \in [0,T]$, it is well known that $u$ is already in $H^1(0,T;H)$.
Thus the question arises whether $u$ is in $H^1(0,T;H)$ also in the non-autonomous case.
This question seems still to be open and was explicitly asked by Lions \cite[p.\ 68]{Lio61} in the case that $\fra(t,.,.)$ is symmetric for all $t \in [0,T]$.
We say that $\fra$ has \emph{maximal regularity} in $H$ if 
for all $f \in L^2(0,T;H)$ and $u_0=0$ the solution $u$ of \eqref{eq:CPinV'} is in $H^1(0,T;H)$,
and consequently in
\[
	\MR_\fra(H) := \{ u \in L^2(0,T;V) \cap H^1(0,T;H) : \A u \in L^2(0,T;H) \}.
\]
It is easy to see that $\fra$ has maximal regularity in $H$ implies that the solution $u$ of \eqref{eq:CPinV'} is in $H^1(0,T;H)$ for every $f \in L^2(0,T;H)$ and $u_0 \in Tr_\fra$,
where $Tr_\fra := \{v(0) : v \in \MR_\fra(H)\}$.

In the present article the contribution to this question is the following.
Assume additionally that $\fra(t,.,.)$ is symmetric for all $t \in [0,T]$
and of \emph{bounded variation}; i.e., there exists a bounded and non-decreasing function $g \colon [0,T] \to \R$ such that
\begin{equation*}
	\abs{\fra(t,v,w)- \fra(s,v,w)} \le [g(t)-g(s)] \norm{v}_V \norm{w}_V \quad (0\le s \le t \le T, v,w \in V).
\end{equation*}
Then $\fra$ has maximal regularity in $H$ and $Tr_\fra= V$. Moreover $\MR_\fra(H)$ is continuously embedded in $C([0,T];V)$ (see Theorem~\ref{thm:MRp1}).
The fact that the solution is continuous with values in $V$ is not obvious at all and plays a central role in the following results.
In Theorem~\ref{thm:MR} we extend this regularity result to certain perturbations of $\A$, including multiplicative perturbations 
(see Corollary~\ref{cor:pert}).
We obtain this result by establishing refined product rules for functions in the maximal regularity space
\(
	\MR_\fra(H),
\)
which are of independent interest.
For example to obtain a priori estimates for semilinear or quasilinear problems (see \eqref{eq:NLCP}).

The question of $H$-maximal regularity is important for several reasons.
First of all, if Robin boundary conditions are considered, only the operator $A(t)$ associated to $\fra(t,.,.)$ on $H$ realizes these boundary conditions.
The main reason for studying this problem is the importance for non-linear problems.
They are mainly solved by applying the Banach or the Schauder fixed point Theorem.
For that a suitable invariant space is needed and this may be the space $\MR_\fra(H)$ if maximal regularity in $H$ is valid.
In addition, if the injection of $V$ in $H$ is compact, then the injection of $\MR_\fra(H)$ in $L^2(0,T;H)$ is compact (see Theorem~\ref{thm:A-L}).
This allows one to use Schauder's or more appropriately Schaefer's fixed point theorem.
This had be done in this context in \cite{AC10},
where an isotropic quasilinear parabolic problem of the form
\begin{equation*}
		\left\{  \begin{aligned}
          		&u'  + m(t,x,u,\nabla u) \A u=  f(t,x,u,\nabla u)\\
                       &u(0) =u_0,
		\end{aligned} \right.
\end{equation*}
where $\A$ is a time independent operator, was investigated.
With our new results we now obtain an analogous result for  time dependent $\A$.
For this we need a non-autonomous Aubin--Lions lemma which we prove in Section~6.

We now comment on the relation of our investigations with known results.
Our results improve the results of \cite{ADLO14} where Lipschitz continuity of $\fra(.,v,w)$ for all $v,w \in V$ was assumed
whereas we only need bounded variation. 
On the other hand we restrict ourselves to symmetric forms whereas \cite{ADLO14} only the uniform square root property was assumed.
The method we use here is completely different than the one of \cite{ADLO14},
where a suitable similarity transformation is used which allows one to reduce the problem to Lions' result Theorem~\ref{thm:Lions}. 

Lions himself proved maximal regularity in $H$ if $\fra(t,.,.)$ is symmetric for all $t \in [0,T]$ and $\fra(.,v,w) \in C^1([0,T])$ for all $v,w \in V$ (see \cite[p.\ 65]{Lio61}).
He also proved the following: if $f \in H^1(0,T;H)$ and $u_0 \in D(A(0))$, $\fra(t,.,.)$ is symmetric for all $t \in [0,T]$ and $\fra(.,v,w) \in C^2([0,T])$ for all $v,w \in V$, then the solution $u$ of \eqref{eq:CPinV'} is in $H^1(0,T;H)$ (see \cite[p.\ 94]{Lio61}).

Bardos generalized Lions' result in \cite{Bar71}. He proved maximal regularity in $H$ and $Tr_\fra=V$, 
under the assumptions that the domains of both $A(t)^{1/2}$ and 
$A(t)^{*1/2}$ coincide with $V$ as spaces and topologically with constants independent of $t$,
and that $\A(.)^{1/2}$ is continuously differentiable with values in $\L(V,V')$.

With a different approach, maximal regularity in $H$ was shown in \cite{OS10}, if there exist some constants $L$ and $\alpha > \frac 1 2$ such that
\[
	\abs{\fra(t,v,w)-\fra(s,v,w)} \le L \abs{t-s}^\alpha \norm{v}_V \norm{w}_V \quad (s,t \in [0,T], v,w \in V).
\]
This result was improved in \cite{HO14} in the following way. If $\fra$ satisfies some ``Dini'' condition, which is a generalization of the Hölder continuity above, 
then $\fra$ has maximal regularity in $H$ and $Tr_\fra = D(A(0)^{1/2})$.

More recent further contributions to maximal regularity for non-autonomous problems are
\cite{ADO14}, \cite{ADK14}, \cite{ACFP07}, \cite{SP01}, \cite{Ama04}.

The article is organized as follows.
Section~2 has preliminary character. There we give precise definitions and introduce some notation.
The tool kit (Proposition~\ref{prop:Lip_cont} and Proposition~\ref{prop:BV}) for the main results is produced in Section 3.
In Section~4 we obtain Theorem~\ref{thm:MRp1}  by regularization of the form in time.
A perturbation result in Section~5 will broaden the spectrum of applications.
In Section~6 we prove an Aubin--Lions lemma.
We illustrate our abstract results in Section~7 by some applications to elliptic operators and show existence for a quasi-linear problem.
\section*{Acknowledgment}
The author thanks Wolfgang Arendt for many helpful discussions. 
He gratefully acknowledges the financial support of the DFG Graduate School 1100: Modeling, Analysis and Simulation in Economathematics.

\section{Non-autonomous forms}\label{sec:forms}
Let $\K$ be the field $\R$ or $\C$ and let $V, H$ be Hilbert spaces over $\K$,
such that $V \overset d \hookrightarrow H$; i.e., $V$ is continuously and densely embedded in $H$. 
Then there exists a constant $c_H$ such that
\begin{equation}\label{eq:VinH}
	\norm{v}_H \le c_H \norm{v}_V \quad (v \in V).
\end{equation}
We denote by $V'$ the antidual (or dual if $\K=\R$) of $V$, and by $\langle ., .\rangle$ the duality between $V'$ and $V$. Furthermore we embed $H$ into $V'$ by the mapping
\[
	u \mapsto \left[ v \mapsto (u \mid v)_H \right].
\]
Then $(u \mid v)_H = \langle u, v \rangle$ for all $u \in H$ and $v \in V$, $H$ is dense in $V'$ and 
\[
	\norm{u}_{V'} \le c_H \norm u_H \quad (u \in H)
\]
where $c_H$ is the same constant as in \eqref{eq:VinH}.

Let $T>0$. The mapping
\[
	\fra \colon [0,T] \times V \times V \to \K
\]
is called a \emph{non-autonomous form} if $\fra(t,.,.)$ is sesquilinear for all $t \in [0,T]$ and
$\fra(.,u,v)$ is measurable for all $u,v \in V$.

We say the non-autonomous form $\fra$ is \emph{$V$-bounded} if there exists a constant $M$ such that
\begin{equation}\label{eq:Vbounded}
	\abs{\fra(t,v,w)} \le M \norm v_V \norm w_V \quad (t \in [0,T], v,w \in V),
\end{equation}
and \emph{coercive} if there exists an $\alpha >0$ such that
\begin{equation}\label{eq:coercive}
	\Re \fra(t,v,v) \ge \alpha\norm{v}_V^2 \quad (t \in [0,T], v \in V).
\end{equation}
A non-autonomous form $\fra$ is called \emph{symmetric} if
\[
	\fra(t,v,w) = \overline{\fra(t,w,v)} \quad (t \in [0,T], v,w \in V).
\]
Furthermore we say the non-autonomous form $\fra$ is \emph{Lipschitz continuous} if there exists a constant $L$ such that
\begin{equation}\label{eq:formderivative}
	\abs{\fra(t,v,w)- \fra(s,v,w)} \le L \abs{t-s} \norm{v}_V \norm{w}_V \quad (t,s\in [0,T], v,w \in V)
\end{equation}
and of \emph{bounded variation} if
there exists a bounded and non-decreasing function $g \colon [0,T]\to \R$ such that
\begin{equation}\label{eq:Mpunkt}
	\abs{\fra(t,v,w)- \fra(s,v,w)} \le [g(t)-g(s)] \norm{v}_V \norm{w}_V \quad (0\le s \le t \le T, v,w \in V). 
\end{equation}

Let $\fra$ be a $V$-bounded and coercive non-autonomous form. Then for fixed $t \in [0,T]$ there exists an invertible operator $\A(t) \in \L(V,V')$ such that
\[
	\langle \A(t)u, v \rangle = \fra(t,u,v) \quad (u,v \in V)
\]
by \eqref{eq:Vbounded}, \eqref{eq:coercive} and the Lax--Milgram theorem.
We consider $\A$ as the multiplication operator from $L^2(0,T;V)$ to $L^2(0,T;V')$ and say that $\A$ is the \emph{associated operator} of $\fra$, or $\A \sim \fra$.
Further we define the \emph{maximal regularity space} of $\fra$ by
\[
	\MR_\fra(H) := \{ u \in L^2(0,T;V) \cap H^1(0,T;H) : \A u \in L^2(0,T;H) \},
\]
equipped with the norm
\[
	\norm{u}_{\MR_\fra(H)}^2 := \norm{\A u}_{L^2(0,T;H)}^2 + \norm{u'}_{L^2(0,T;H)}^2.
\]
Note that $\MR_\fra(H)$ is a Hilbert space and $\MR_\fra(H)$ is continuously embedded in $L^2(0,T;V)$.
Moreover for $u\in \MR_\fra(H)$ we always identify $u$ by its continuous version on $C([0,T];V)$.

Let $\fra$ be a $V$-bounded, coercive non-autonomous form of bounded variation, where $g \colon [0,T]\to \R$ is bounded and non-decreasing such that \eqref{eq:Mpunkt} holds.
Then we may define the right-continuous versions $g^+$ and $\fra^+$ of $g$ and $\fra$ (here we set $g^+(T)=\lim_{t\uparrow T}g(t)$ and $\fra^+(T,.,.)=\lim_{t\uparrow T}\fra(t,.,.)$)
and the left-continuous versions $g^-$ and $\fra^-$ of $g$ and $\fra$ (here we set $g^-(0)=\lim_{t\downarrow 0}g(t)$ and $\fra^-(0,.,.)=\lim_{t\downarrow 0}\fra(t,.,.)$).
Then $\MR_\fra(H) = \MR_{\fra^+}(H) =\MR_{\fra^-}(H)$ and if $\A\sim \fra$, $\A^+\sim \fra^+$ and $\A^- \sim \fra^-$ we have $\A u = \A^+ u = \A^- u$ in $L^2(0,T;V')$ for all $u \in L^2(0,T;V)$.

Finally we denote by $\mu_g$ the unique Borel measure on $[0,T]$ which is defined by
\[
	\mu_g((a,b]) = g^+(b) -g^+(a) \quad (0 \le a < b \le T).
\]

\section{A differentiation formula}\label{sec:diffformula}
Let $V, H$ be Hilbert spaces over the field $\K$ such that $V \overset d \hookrightarrow H$.
Furthermore let
\[
	\fra \colon [0,T]\times V\times V \to \K
\]
be a symmetric, $V$-bounded, coercive non-autonomous form and $\A \sim \fra$.
The purpose of this section is to obtain properties of the maximal regularity space $\MR_\fra(H)$ 
and the function $t \mapsto \fra(t,u(t),u(t))$ for $u \in \MR_\fra(H)$. 

\begin{proposition}\label{prop:diff_formula}
	Let $u\in \MR_\fra(H)$. Then
	\begin{multline*}
		- \int _0^T \fra(t,u,u) \varphi' \ \d{t} = \int_0^T 2 \Re (\A u\mid u')_H \varphi \ \d{t}\\
			+ \lim_{h\to 0} \frac 1 h \int_0^T \left[\fra(t+h,u(t+h),u(t))-\fra(t,u(t+h),u(t))\right] \varphi(t) \ \d{t}
	\end{multline*}
	for every $\varphi \in C_c^1(0,T)$.
\end{proposition}
\begin{proof}
	Let $u\in \MR_\fra(H)$, $\varphi \in C_c^1(0,T)$ and $h \in \R$ such that $\supp(\varphi) \pm h \subset [0,T]$.
	Then
	\begin{align*}
		&-\int_0^T \fra(t,u(t),u(t)) \frac{\varphi(t-h)-\varphi(t)}{-h} \ \d{t}\\
		&= \int_0^T \frac 1 h \left[ \fra(t+h,u(t+h),u(t+h))-\fra(t,u(t),u(t)) \right] \varphi(t) \ \d{t}\\
		&= \int_0^T \frac 1 h [ \fra(t+h,u(t+h),u(t+h)-u(t))+ \fra(t,u(t+h)-u(t),u(t))\\
			 &\quad\qquad+ \fra(t+h,u(t+h),u(t)) - \fra(t,u(t+h),u(t))  ] \varphi(t) \ \d{t}\\
		&= \int_0^T  \big[ (\A(t+h)u(t+h) \mid \tfrac 1 h \textstyle\int_t^{t+h} u'(s) \, \d{s})_H\\
		 	&\quad\qquad+ (\tfrac 1 h \textstyle\int_t^{t+h} u'(s) \, \d{s} \mid \A(t)u(t) )_H \big] \varphi(t) \ \d{t}\\
			 &\quad+ \frac 1 h \int_0^T \big[\fra(t+h,u(t+h),u(t)) - \fra(t,u(t+h),u(t))  \big] \varphi(t) \ \d{t}.
	\end{align*}
	Since $u \in L^2(0,T;V)$ we have by \eqref{eq:Vbounded} that $\fra(.,u,u) \in L^1(0,T)$. 
	Hence the left hand side converges to
	$- \int _0^T \fra(t,u,u) \varphi' \, \d{t}$ as $h\to 0$.
	Moreover, since $\A u \in L^2(0,T;H)$ we have $\A u(.+h) \to \A u$ in $L^2(0,T;H)$ as $h \to 0$ 
	and since $u \in H^1(0,T;H)$ we have $\frac 1 h \int_.^{.+h} u'(s) \, \d{s} \to u'$ in $L^2(0,T;H)$ as $h \to 0$.
	Thus the first integral on the right hand side converges to $\int_0^T 2 \Re (\A u\mid u')_H \varphi \, \d{t}$.
\end{proof}
Before we come to the main results of this section we need the following three lemmas.

\begin{lemma}\label{lem:weak_cont}
	Let $u \in C(0,T;H) \cap L^\infty(0,T;V)$. Then $u(.)$ is weakly continuous in $V$.
\end{lemma}
\begin{proof}
	Let $N \subset [0,T]$ be a null set such that $\norm{u(t)}_V \le \norm{u}_{L^\infty(0,T;V)}$ holds for all $t\in [0,T]\setminus N$.
Let $t \in [0,T]$ and $(t_n) \subset [0,T]\setminus N$ with $t_n \to t$. 
Since $(u(t_n))$ is bounded in $V$, every subsequence of $(u(t_n))$ has a weakly convergent subsequence converging to some $v$ in $V$.
Since $u(.)$ is continuous in $H$ we obtain $v = u(t)$, hence $u(t_n) \rightharpoonup u(t)$ as $n \to \infty$.
This implies $\norm{u(t)}_V \le \norm{u}_{L^\infty(0,T;V)}$ for all $t\in [0,T]$. Hence we can choose $N = \emptyset$ and therefore $u(.)$ is weakly continuous in $V$.
\end{proof}

\begin{lemma}\label{lem:Lebesguepoint}
	Suppose $\fra$ is additionally right-continuous (left-continuous); i.e., for every $t \in [0,T)$ ($t \in (0,T]$), $\varepsilon >0$ there exists $\delta >0$ such that
	\begin{equation}\label{eq:form_cont}
		\abs{\fra(t,v,w) - \fra(s,v,w)} \le \varepsilon \norm{v}_V \norm{w}_V \quad (v,w \in V)
	\end{equation}
	for all $s \in [0,T]$ with $t \le s \le t +\delta$ ($t-\delta \le s \le t$). Then for $u \in \MR_\fra(H) \cap L^\infty(0,T;V)$
	\[
		\frac 1 h \int_t^{t+h} \abs{\fra(s,u(s),u(s))- \fra(t,u(t),u(t))} \ \d{s} \to 0
	\]
	as $h \downarrow 0$ ($h \uparrow 0$).
\end{lemma}
Note that for $u \in \MR_\fra(H)$ we always consider the representative which is in $C([0,T];H)$.
\begin{proof}
	We prove the statement for the case that $\fra$ is right-continuous, the other case is similar.
	Let $u \in \MR_\fra(H) \cap L^\infty(0,T;V)$, $\varepsilon >0$ and $t \in [0,T)$. 
	Since $\fra$ is right-continuous, there exists $\delta >0$ such that \eqref{eq:form_cont} holds for all $t \le s \le t +\delta$.
	Thus for $0<h<\min\{\delta, T-t\}$
	\begin{align*}
		&\frac 1 h \int_t^{t+h} \abs{\fra(s,u(s),u(s))- \fra(t,u(t),u(t))} \ \d{s}\\
		&\le\frac 1 h \int_t^{t+h} \abs{\fra(s,u(s),u(s)-u(t))}+ \abs{\fra(t,u(s)-u(t),u(t))}\\
		&\quad\qquad + \abs{\fra(s,u(s),u(t))- \fra(t,u(s),u(t))}  \ \d{s}\\
		&\le \frac 1 h \int_t^{t+h} \int_t^s \abs{(\A(s)u(s) \mid u'(r))_H} \ \d{r} \ \d{s}
			+\frac 1 h \int_t^{t+h} \abs{\fra(t,u(s)-u(t),u(t))} \ \d{s}\\
		&\quad + \varepsilon \frac 1 h \int_t^{t+h} \norm{u(s)}_V\norm{u(t)}_V \ \d{s}\\
		&\le \frac 1 h \int_t^{t+h} \norm{\A(s)u(s)}_H \ \d{s} \int_t^{t+h} \norm{u'(r)}_H \ \d{r} \\
		&\quad +\frac 1 h \int_t^{t+h} \abs{\fra(t,u(s)-u(t),u(t))} \ \d{s}+ \varepsilon \norm{u}_{L^\infty(0,T;V)}^2.
	\end{align*}
	Finally, the second line from below is dominated by 
	\[
		\left( \int_t^{t+h} \norm{\A(s)u(s)}_H^2 \ \d{s} \int_t^{t+h} \norm{u'(r)}_H^2 \ \d{r} \right)^{1/2}
	\]
	by Hölder's inequality. Moreover, the function $s \mapsto \abs{\fra(t,u(s)-u(t),u(t))}$ is continuous, since $u(.)$ is weakly continuous in $V$ by Lemma~\ref{lem:weak_cont}.
	Hence taking the limit $h \downarrow 0$ proves the claim.
\end{proof}

\begin{lemma}\label{lem:nullset}
	Suppose $\fra$ is additionally Lipschitz continuous and $\tilde V$ is a separable subspace of $V$. 
	Then there exists a Lebesgue null set $N \subset [0,T]$ such that 
	$\lim_{h\to 0} \frac 1 h [\fra(t+h,v,v)-\fra(t,v,v)]$ exists for all $t \in [0,T] \setminus N$ and all $v \in \tilde V$.
\end{lemma}
\begin{proof}
	Let $L$ be a constant such that \eqref{eq:formderivative} holds and let $\{ v_n : n \in \N \}$ be a dense subset of $\tilde V$.
	By \eqref{eq:formderivative} for every $n \in \N$ there exists a null set $N_n \subset [0,T]$ such that $\fra(t,v_n,v_n)$ is
	differentiable for all $t \in [0,T] \setminus N_n$ with
	\[
		\abs*{\tfrac{\d}{\d t}\fra(t,v_n,v_n)} \le L \norm{v_n}_V^2 \quad (t \in [0,T] \setminus N_n).
	\]
	Hence $N := \cup_{n \in \N} N_n$ is the desired null set.
\end{proof}

The following two propositions are essential tools for the next two sections.

\begin{proposition}\label{prop:Lip_cont}
Suppose $\fra$ is additionally Lipschitz continuous.
Then 
\[
	\MR_\fra(H) \hookrightarrow C([0,T];V)
\] 
and for $u \in \MR_\fra(H)$ we have $\fra(.,u,u) \in W^{1,1}(0,T)$ with
\begin{equation}\label{eq:diff_formula_Lipschitz}
	(\fra(.,u,u))'= 2\Re(\A u\mid u')_H + \fra'(.,u,u),
\end{equation}
where $\fra'(t,v,v)$ equals $\lim_{h\to 0} \frac 1 h [\fra(t+h,v,v)-\fra(t,v,v)]$ if the limit exists and $0$ otherwise, for all $t \in [0,T]$ and all $v \in V$.
\end{proposition}
\begin{proof}
	Let $u \in \MR_\fra(H)$, $\varphi \in C^1_c(0,T)$ and $h \in \R$ such that $\supp\varphi \pm h \subset [0,T]$.
	We first show $\fra(.,u,u) \in W^{1,1}(0,T)$ such that \eqref{eq:diff_formula_Lipschitz} holds.
	By Proposition~\ref{prop:diff_formula} it suffices to show
	\begin{equation*}
		\frac 1 h \int_0^T \left[\fra(t+h,u(t+h),u(t))-\fra(t,u(t+h),u(t))\right] \varphi \ \d{t} \to \int_0^T \fra'(t,u,u) \varphi \ \d{t}
	\end{equation*}
	as $h\to0$. 
	Note that by the Lipschitz continuity of $\fra$ and the convergence of $u(.+h) \to u$ in $L^2(\supp(\varphi);V)$ as $h\to 0$ it suffices to show
	\begin{equation*}
		\frac 1 h \int_0^T \left[\fra(t+h,u(t),u(t))-\fra(t,u(t),u(t))\right] \varphi \ \d{t} \to \int_0^T \fra'(t,u,u) \varphi \ \d{t}
	\end{equation*}
	as $h\to0$.
	We have on $\supp(\varphi)$
	\[
		\frac 1 h \left[\fra(.+h,u(.),u(.))-\fra(.,u(.),u(.)) \right] \le L \norm{u(.)}_V^2 \in L^1(0,T).
	\]
	Moreover since $u\in L^2(0,T;V)$ the subspace $\operatorname{span}\{u(t)\}_{t\in[0,T]}$ is separable. 
	Thus by Lemma~\ref{lem:nullset}  there exists a null set $N\subset[0,T]$ such that
	\[
		\frac 1 h \left[\fra(t+h,u(t),u(t))-\fra(t,u(t),u(t)) \right] \to \fra'(t,u(t),u(t)) \quad (h\to0)
	\] 
	for all $t \in (0,T)\setminus N$.
	Now the claim follows by the dominated convergence theorem.
	
	Since the function $\fra(.,u,u)$ is in $W^{1,1}(0,T)$ it has a continuous version.
	By \eqref{eq:coercive} we have $u \in L^\infty(0,T;V)$.
	Since $u \in H^1(0,T;H)$ we may choose the representative of $u$ that is in $C([0,T];H)$.
	Thus by Lemma~\ref{lem:weak_cont} $u(.)$ is weakly continuous in $V$.
	Given this representative of $u$ we obtain by Lemma~\ref{lem:Lebesguepoint} that the function $\fra(.,u,u)$ is continuous.
	Finally let $s,t \in [0,T]$. Then
	\begin{align*}
		\alpha\norm{u(t)-u(s)}_V^2 &\le \fra(t,u(t)-u(s),u(t)-u(s))\\
		 &= 2 \Re \fra(t,u(t)-u(s),u(t)) + \fra(t,u(s),u(s)) - \fra(s,u(s),u(s)) \\
		 	&\quad+ \fra(s,u(s),u(s)) - \fra(t,u(t),u(t))\\
		 &\le 2 \Re \fra(t,u(t)-u(s),u(t)) + L \abs{t-s} \norm{u}_{L^\infty(0,T)}^2\\
		  	&\quad+ \fra(s,u(s),u(s)) - \fra(t,u(t),u(t))\\
		& \to 0 \quad (s \to t). \tag*\qedhere
	\end{align*}
\end{proof}

\begin{lemma}\label{lem:fub}
	Let $g \colon [0,T] \to \R$ be bounded and non-decreasing and let $\psi \in C_c(0,T)$. Then
	\[
		\lim_{h \to 0} \int_{[0,T]} \frac{g(t+h)-g(t)}{h} \psi (t) \ \d{t} = \int_{[0,T]} \psi \ \d{\mu_g},
	\]
	where $\mu_g$ denotes the unique Borel measure on $[0,T]$ such that $\mu((a,b]) = g^+(b)-g^+(a)$ for $0\le a<b \le T$.
\end{lemma}
\begin{proof}
	We extend $g$ to $\R$ by $g(0)$ on $(-\infty, 0)$ and by $g(T)$ on $(T,\infty)$.
	Moreover we extend $\psi$ to $\R$ by $0$ on the complement of $[0,T]$.
	Note that $g$ has at most countably many discontinuities.
	Let $h >0$. Then
	\begin{align*}
		\int_\R \frac{g(t+h)-g(t)}{h} \psi (t) \ \d{t}
			&= \int_\R \frac{g^+(t+h)-g^+(t)}{h} \psi (t) \ \d{t}\\
			&= \frac 1 h\int_\R \int_\R \1_{(t,t+h]}(s) \psi (t) \ \d \mu_g(s)\ \d{t}\\
			&= \frac 1 h \int_\R \int_\R \1_{[s-h,s)}(t) \psi (t) \ \d{t} \ \d \mu_g(s)
	\end{align*}
	by Fubini's theorem.
	Since $\psi$ is uniformly continuous we have $\frac 1 h \int_{s-h}^s \psi (t) \ \d{t} \to \psi(s)$ uniformly as $h \to 0$.
	Hence by the dominated convergence theorem
	\[
		\int_\R \frac{g(t+h)-g(t)}{h} \psi (t) \ \d{t} \to \int_\R \psi \ \d{\mu_g}
	\]
	as $h\to 0$. The case $h < 0$ is similarly.
\end{proof}

\begin{proposition}\label{prop:BV}
	Suppose $\fra$ is additionally of bounded variation.
	Then 
	\[
		\MR_\fra(H)\cap L^\infty(0,T;V) \hookrightarrow C([0,T];V)
	\]
	and for $u \in \MR_\fra(H)\cap L^\infty(0,T;V)$ we have $\fra(.,u,u) \in BV(0,T)$ and
	\begin{multline}\label{eq:BVformula}
		\fra^-(t,u(t),u(t))- \fra^+(0,u(0),u(0))\\ \le \int_0^t 2 \Re (\A u \mid u')_H \ \d{s} + \int_{(0,t)} \norm{u}_V^2 \ \d{\mu_g}.
	\end{multline}
\end{proposition}
Note that we will see in the next section that $\MR_\fra(H) \hookrightarrow L^\infty(0,T;V)$.
\begin{proof}
	Let $u \in \MR_\fra(H)\cap L^\infty(0,T;V)$.
	First we show that $u \in C([0,T];V)$.
	For $\varphi \in C^{1}_c(0,T)$ we have by \eqref{eq:Mpunkt} and Lemma~\ref{lem:fub}
	\begin{align*}
		&\lim_{h\to 0}\frac 1 h \int_0^T \left[\fra(t+h,u(t+h),u(t))-\fra(t,u(t+h),u(t))\right] \varphi \ \d{t}\\
		&\le \limsup_{h\to 0}\int_0^T \frac{g(t+h)-g(t)}h \abs{\varphi} \ \d{t}\ \norm{u}_{L^\infty(0,T;V)}^2\\
		&= \int_{[0,T]} \abs{\varphi} \ \d{\mu_g}\ \norm{u}_{L^\infty(0,T;V)}^2,
	\end{align*}
	where $g \colon[0,T] \to \R$ is a bounded and non-decreasing function such that \eqref{eq:Mpunkt} holds.
	Thus by Proposition~\ref{prop:diff_formula} we obtain
	\begin{equation}\label{eq:phist}
		-\int_0^T \fra(t,u,u) \varphi' \, \d{t} \le  \int_0^T 2 \Re (\A u\mid u')_H \varphi \ \d{t} + \int_{[0,T]} \abs{\varphi} \ \d{\mu_g}\ \norm{u}_{L^\infty(0,T;V)}^2
	\end{equation}
	for all $\varphi \in C^{1}_c(0,T)$ and by density even for Lipschitz continuous $\varphi$ with $\varphi(0)=\varphi(T)=0$.
	Let $0 \le t < s < T$, $0<\delta < \min\{T-s, \frac {s-t} 2  \}$ and set
	\[
		\varphi(r) = \frac {r-t} \delta \1_{(t,t+\delta)}(r) + \1_{[t+\delta, s]}(r) + \frac {s+\delta-r}\delta \1_{(s, s+\delta)}.
	\]
	We insert this $\varphi$ in \eqref{eq:phist} and take the limit $\delta \to 0$. Hence by Lemma~\ref{lem:Lebesguepoint}
	\[
		\fra^+(s,u(s),u(s))- \fra^+(t,u(t),u(t)) \le \int_t^s 2(\A u \mid u')_H \ \d{r} + [g^+(s)-g^+(t)] \norm{u}_{L^\infty(0,T;V)}^2
	\] 
	 for $0 \le t < s < T$.
	By Lemma~\ref{lem:weak_cont} the function $u(.)$ is weakly continuous in $V$.
	Now let $s,t \in [0,T)$ with $s > t$. Then
	\begin{align*}
		\alpha\norm{u(t)-u(s)}_V^2 &\le \fra^+(t,u(t)-u(s),u(t)-u(s))\\
		 &= 2 \Re \fra^+(t,u(t)-u(s),u(t))\\
		 	&\quad + \fra^+(t,u(s),u(s)) - \fra^+(s,u(s),u(s)) \\
		 	&\quad+ \fra^+(s,u(s),u(s)) - \fra^+(t,u(t),u(t))\\
		 &\le 2 \Re \fra^+(t,u(t)-u(s),u(t)) + [g^+(s)-g^+(t)] \norm{u}_{L^\infty(0,T;V)}^2\\
		  	&\quad+ \int_t^s 2\abs{(\A u \mid u')_H} \ \d{r} + [g^+(s)-g^+(t)] \norm{u}_{L^\infty(0,T;V)}^2\\
		& \to 0 \quad (s \downarrow t).
	\end{align*}
	Hence $u(.)$ is right-continuous in $V$. Similarly $u(.)$ is left-continuous in $V$.
	
	It remains to show \eqref{eq:BVformula}. Let $\varphi \in C_c^1(0,T)$.
	Similarly as in the first estimate of the proof we have for sufficiently small $h \in \R$
	\begin{align*}
		&\frac 1 h \int_0^T \left[\fra(t+h,u(t+h)-u(t),u(t))-\fra(t,u(t+h)-u(t),u(t))\right] \varphi \ \d{t}\\
		&\le \int_{[0,T]} \abs{\varphi} \ \d{\mu_g}\ \norm{u(.+h)-u}_{L^\infty(\abs{h},T-\abs{h};V)} \norm{u}_{L^\infty(0,T;V)}.
	\end{align*}
	We obtain by the continuity of $u(.)$ in $V$, the above estimate and \eqref{eq:Mpunkt}
	\begin{align*}
		&\lim_{h\to 0}\frac 1 h \int_0^T \left[\fra(t+h,u(t+h),u(t))-\fra(t,u(t+h),u(t))\right] \varphi \ \d{t}\\
		&=\lim_{h\to 0}\frac 1 h \int_0^T \left[\fra(t+h,u(t),u(t))-\fra(t,u(t),u(t))\right] \varphi \ \d{t}\\
		&\le \limsup_{h\to 0}\int_0^T \frac{g(t+h)-g(t)}h \norm{u}^2_V \abs\varphi \ \d{t}\ \\
		&= \int_{[0,T]} \norm{u}^2_V \abs{\varphi} \ \d{\mu_g},
	\end{align*}
	where we used Lemma~\ref{lem:fub} in the last step. 
	Thus by Proposition~\ref{prop:diff_formula}
	\begin{equation}\label{eq:BVestimate2}
		- \int _0^T \fra(t,u,u) \varphi' \ \d{t} \le \int_0^T 2 \Re (\A u\mid u')_H \varphi \ \d{t}
			+ \int_{[0,T]} \norm{u}^2_V \abs{\varphi} \ \d{\mu_g}.
	\end{equation}
	for all $\varphi \in C_c^1(0,T)$ and by density even for Lipschitz continuous $\varphi$ with $\varphi(0)=\varphi(T)=0$.
	Let $t \in [0,T]$, $0<\delta < \min\{T-t, \frac t 2  \}$ and set
	\[
		\varphi(s) = \frac s \delta \1_{(0,\delta)}(s) + \1_{[\delta, t-\delta]}(s) + \frac {t-s}\delta \1_{}(t-\delta, t).
	\]
	If we insert this particular choice of $\varphi$ in \eqref{eq:BVestimate2}, then taking the limit $\delta \downarrow 0$ shows \eqref{eq:BVformula} by Lemma~\ref{lem:Lebesguepoint}.
\end{proof}

\section{Well posedness with maximal regularity}\label{sec:MR1}
Let $V, H$ be separable Hilbert spaces over the field $\K$ such that $V \overset d \hookrightarrow H$.

\begin{theorem}\label{thm:MRp1}
	Let $\fra \colon [0,T] \times V \times V \to \K$ be a $V$-bounded coercive symmetric form of bounded variation and $\A\sim\fra$.
	Then for every $f \in L^2(0,T;H)$ and $u_0 \in V$ there exists a unique $u \in \MR_\fra(H)$ such that
	\begin{equation*}
		\left\{\begin{aligned}
			u' + \A u &= f\\
			u(0)&=u_0.
		\end{aligned} \right.
	\end{equation*}
	Moreover $\MR_\fra(H) \hookrightarrow C([0,T];V)$.
\end{theorem}
\begin{proof}
	Let $g\colon [0,T]\to \R$ be a bounded and non-decreasing function such that \eqref{eq:Mpunkt} holds.
	We extend $\fra$ to $\R\times V \times V$ by $\fra(0,.,.)$ for $t<0$ and by $\fra(T,.,.)$ for $t>T$ 
	and we extend $g$ to $\R$ by $g(0)$ for $t<0$ and by $g(T)$ for $t>T$.
	We let $\rho \colon \R \to [0,\infty)$ be a mollifier with support $[-1,1]$ and define the function $\rho_n \colon \R \to [0,\infty)$ by $\rho_n(t):= n \rho(n t)$ for $n \in \N$.
	Furthermore we define the form $\fra_n \colon [0,T] \times V \times V \to \K$ by $\fra_n(t,u,v) := (\fra(.,u,v) * \rho_n)(t)$
	and the function $g_n \colon [0,T] \to \R$ by $g_n := \rho_n * g$ for $n \in \N$.
	Note that $\fra_n$ is a symmetric form with the same $V$-bound and coerciveness constant as $\fra$.
	Moreover $g_n$ is bounded and non-decreasing and $\fra_n$ is of bounded variation where
	\begin{equation*}
		\abs{\fra_n(t,v,w)- \fra_n(s,v,w)} \le [g_n(t)-g_n(s)] \norm{v}_V \norm{w}_V \quad (0\le s \le t \le T, v,w \in V).
	\end{equation*}
	We denote by $\A_n$ the associated operator of $\fra_n$. 
	
	Let $f \in L^2(0,T;H)$ and $u_0 \in V$.
	By Theorem~\ref{thm:Lions} there exists a unique $u \in \MR(V,V')$ such that
	\begin{equation*}
		\left\{\begin{aligned}
			u' + \A u &= f\\
			u(0)&=u_0.
		\end{aligned} \right.
	\end{equation*}
	By a combination of \cite[Theorem~1.1, p.~129]{Lio61} and \cite[Theorem~5.1, p.~138]{Lio61} and Proposition~\ref{prop:Lip_cont}
	(see also \cite[Thoerem~4.2]{ADLO14}) for every $n \in \N$ there exists a unique $u_n \in \MR_{\fra_n}(H)$ such that
	\begin{equation*}
		\left\{\begin{aligned}
			u'_n + \A_n u_n &= f\\
			u_n(0)&=u_0
		\end{aligned} \right.
	\end{equation*}
	and $u_n \in C([0,T];V)$.
	It is our aim to show that $u_n$ converges to $u$ in $\MR(V,V')$ and converges weakly to $u$ in $H^1(0,T;H)$; hence $u \in \MR_\fra(H)$ is the desired solution.
	
	First we provide an estimate for $u_n$. Let $n \in \N$ and $t \in [0,T]$. Then by Proposition~\ref{prop:Lip_cont} and the continuity of $u_n(.)$ in $V$
	\begin{align*}
		&\norm{f}_{L^2(0,t;H)}^2+M\norm{u_0}_V^2 \\
		&\ge \norm{u'_n}_{L^2(0,t;H)}^2 + \norm{\A_n u_n}_{L^2(0,t;H)}^2\\
			&\quad + 2\Re(\A_n u_n, u'_n)_{L^2(0,t;H)} + \fra_n(0,u_n(0),u_n(0))\\
		&= \norm{u'_n}_{L^2(0,t;H)}^2 + \norm{\A_n u_n}_{L^2(0,t;H)}^2 \\
			&\quad + \fra_n(t,u_n(t),u_n(t))- \int_0^t \fra'_n(s,u_n,u_n) \ \d{s}
	\end{align*}
	where $\fra'_n(t,v,w) = \rho_n' *\fra(.,v,w) (t)$ for all $v,w \in V$.
	Let $h \in (-\frac t 2,\frac t 2)$, then
	\begin{align*}
		&\frac 1 h \int_{\abs{h}}^{t-\abs{h}} [\fra_n(s+h,u_n,u_n)-\fra_n(s,u_n,u_n)] \ \d{s}\\
		&= \frac 1 h \int_{\abs{h}}^{t-\abs{h}} \int_\R [\fra(r+h,u_n(s),u_n(s))-\fra(r,u_n(s),u_n(s))] \rho_n(s-r) \ \d{r} \ \d{s}\\
		&\le \frac 1 h \int_{\abs{h}}^{t-\abs{h}} \int_\R [g(r+h)-g(r)] \rho_n(s-r) \ \d{r} \ \norm{u_n(s)}_V^2 \ \d{s}\\
		&\le \frac 1 h \int_{\abs{h}}^{t-\abs{h}} [g_n(s+h)-g_n(s)] \norm{u_n(s)}_V^2 \ \d{s}.
	\end{align*}
	Taking the limit $h \to 0$ yields
	\begin{equation*}
		\int_0^t \fra'_n(s,u_n,u_n) \ \d{s} \le \int_{0}^{t} g_n'(s) \norm{u_n(s)}_V^2 \ \d{s}.
	\end{equation*}
	We obtain
	\[
		\alpha \norm{u_n(t)}_V^2 \le \norm{f}_{L^2(0,t;H)}^2+M\norm{u_0}_V^2 + \int_{0}^{t} g_n'(s) \norm{u_n(s)}_V^2 \ \d{s}.
	\]
	for all $t \in [0,T]$.
	Thus by Gronwall's lemma we obtain
	\begin{equation}\label{eq:boundedinV}
	\begin{aligned}
		\tfrac 1 \alpha\norm{u_n(t)}_V^2 &\le \left[ \norm{f}_{L^2(0,t;H)}^2+M\norm{u_0}_V^2 \right] \exp\left\{ \tfrac 1 \alpha \norm{g_n'}_{L^1(0,t)} \right\}\\
			&\le \left[ \norm{f}_{L^2(0,T;H)}^2+M\norm{u_0}_V^2 \right] \exp\left\{ \tfrac 1 \alpha [g(T)-g(0)]\right\}
	\end{aligned}
	\end{equation}
	for all $t \in [0,T]$, since $g_n'(s) \ge 0$ for all $s \in [0,T]$ and
	\begin{align*}
		\int_0^t g'_n(s) \ \d{s} &= \lim_{h \to 0} \frac 1 h\int_0^t \int_\R [g(r+h)-g(r)] \rho_n(s-r) \ \d{r} \ \d{s}\\
		&= \lim_{h \to 0} \frac 1 h \int_\R [g(r+h)-g(r)]  \int_0^t \rho_n(s-r)  \ \d{s} \ \d{r}\\
		&\le \lim_{h \to 0} \frac 1 h \int_\R [g(r+h)-g(r)] \ \d{r}\\
		&\le \lim_{h \to 0} \frac 1 h \int_\R \int_{(r,r+h]} \ \d{\mu_g} \ \d{r}\\
		&= \lim_{h \to 0} \frac 1 h \int_\R \int_{s-h}^{s} \ \d{r} \ \d{\mu_g}(s)\\
		&= g(T)-g(0).
	\end{align*}
	Now combining the above estimates yields
	\begin{align*}
		 \norm{u'_n}_{L^2(0,T;H)}^2 + \norm{\A_n u_n}_{L^2(0,T;H)}^2
		\le C \left[ \norm{f}_{L^2(0,T;H)}^2+M\norm{u_0}_V^2 \right].
	\end{align*}
	where $C := 1+ \tfrac 1 \alpha [g(T)-g(0)] \exp\left\{ \tfrac 1 \alpha [g(T)-g(0)]\right\}$.

	Next we show $u_n \to u$ in $\MR(V,V')$. We set $v_n := u-u_n$, then $v_n \in \MR(V,V')$ is the solution of $v_n(0) =0$ and
	\[
		\dot v_n +\A v_n = (\A_n - \A) u_n.
	\]
	By \eqref{eq:MRV'est} it remains to show that $(\A_n - \A) u_n \to 0$ in $L^2(0,T;V')$. 
	Let $v \in L^2(0,T;V)$ with $\norm{v}_{L^2(0,T;V)}=1$. Then
	\begin{align*}
		&\int_0^T\langle (\A_n - \A) u_n, v \rangle \ \d{t}\\
			&= \int_0^T \fra_n(t, u_n, v ) - \fra(t, u_n, v ) \ \d{t}\\
			&= \int_0^T \int_\R [\fra(t-s, u_n, v)-\fra(t, u_n, v)] \ \rho_n(s) \ \d{s} \ \d{t}\\
			&\le \int_0^T \int_\R \abs{g(t)-g(t-s)} \norm{u_n(t)}_V\norm{v(t)}_V \ \rho_n(s) \ \d{s} \ \d{t}\\
			&\le \int_0^T \int_\R [g(t+\tfrac 1 n)-g(t-\tfrac 1 n)] \norm{u_n(t)}_V\norm{v(t)}_V \ \rho_n(s) \ \d{s} \ \d{t}\\
			&= \int_0^T [g(t+\tfrac 1 n)-g(t-\tfrac 1 n)] \norm{u_n(t)}_V\norm{v(t)}_V \ \d{t}\\
			&\le \norm{g(t+\tfrac 1 n)-g(t-\tfrac 1 n)}_{L^2(0,T)} \norm{u_n(t)}_{L^\infty(0,T;V)}.
	\end{align*}
	Now \eqref{eq:boundedinV} and the convergence of $g(.+h) \to g$ as $h \to 0$ in $L^2(0,T)$
	show that $(\A_n - \A) u_n \to 0$ in $L^2(0,T;V')$.
	
	Since $(u_n)$ is bounded in $H^1(0,T;H)$ and since $u_n \to u$ in $\MR(V,V')$ any subsequence of $(u_n)$ has a weakly $H^1(0,T;H)$ convergent subsequence which converges to $u$.
	Hence $u_n \rightharpoonup u$ in $H^1(0,T;H)$. 
	
	Note that $u_n(.) \to u(.)$ uniformly in $H$ since $\MR(V,V') \hookrightarrow C([0,T];H)$.
	Let $t \in [0,T]$. Then $(u_n(t))$ is bounded in $V$ by \eqref{eq:boundedinV}.
	Thus every subsequence of $(u_n(t))$ has an in $V$ weakly convergent subsequence which converges to $u(t)$.
	Hence $u_n(t) \rightharpoonup u(t)$ in $V$ and by \eqref{eq:boundedinV} we have
	\begin{equation}\label{eq:boundedinV2}
		\tfrac 1 \alpha\norm{u(t)}_V^2
			\le \left[ \norm{f}_{L^2(0,T;H)}^2+M\norm{u_0}_V^2 \right] \exp\left\{ \tfrac 1 \alpha [g(T)-g(0)]\right\} 
	\end{equation}
	for all $t \in [0,T]$. 
	Finally \eqref{eq:boundedinV2} implies that that $\MR_\fra(H) \hookrightarrow L^\infty(0,T;V)$.
	Thus by Proposition~\ref{prop:BV} it follows that $\MR_\fra(H) \hookrightarrow C([0,T];V)$.
\end{proof}

\section{A perturbation result}\label{sec:pert}
Let $V, H$ be Hilbert spaces over the field $\K$ such that $V \overset d \hookrightarrow H$.
Furthermore let
\[
	\fra \colon [0,T]\times V\times V \to \K
\]
be a symmetric, $V$-bounded, coercive non-autonomous form and $\A \sim \fra$.
We define the Banach space $W$ by 
\begin{align*}
	W = \{w \in C(0,T;V) : \A w \in L^2(0,T;H)) \}\\
	\norm{w}_W = \norm{w}_{L^\infty(0,T;V)} + \norm{\A w}_{L^2(0,T;H)}.
\end{align*}
Note that by Theorem~\ref{thm:MRp1} we have $\MR_\fra(H) \hookrightarrow W$.

\begin{theorem}\label{thm:MR}
	Let $\B \colon W \to L^2(0,T;H)$ be a bounded operator and let $b$ be a constant such that
	\begin{equation}\label{eq:Bbound}
		\norm{\B}_{\L(W, L^2(0,T;H))} \le b.
	\end{equation}
	Suppose there exist $0 <\delta \le 1$ and a positive Borel measure $\nu$ on $[0,T]$ such that 
	\begin{equation}\label{eq:Bcoercive}
		\int_0^t \Re(\A u+\B u \mid \A u )_H  \ \d{s} \ge \delta \int_0^t \norm{\A u}_H^2 \ \d{s} - \int_{[0,t)} \norm{u}_V^2 \ \d{\nu}
	\end{equation}
	for all $t \in [0,T]$ and $u \in \MR_\fra (H)$.
	Then for every $u_0\in V$, $f \in L^2(0,T;H)$ there exists a unique $u$ in $\MR_\fra(H)$ such that
	\begin{equation*}
		\left\{\begin{aligned}
			u' + \A u +\B u &= f\\
			u(0)&=u_0.
		\end{aligned}\right.
	\end{equation*}
	Moreover there exists a constant $C>0$ depending only on $c_H$, $T$, $\alpha$, $M$, $g(T)-g(0)$, $b$, $\delta$ and $\nu([0,T])$ such that
	\begin{equation}\label{eq:MR_est}
		\norm{u}_{\MR_\fra(H)}^2 \le C \left[\norm{f}_{L^2(0,T;H)}^2 + \norm{u_0}_V^2\right]. 
	\end{equation}
\end{theorem}
\begin{proof}
We use the method of continuity. For $0 \le \lambda \le 1$ consider the mapping 
\[
	\Phi_\lambda \colon \MR_\fra(H) \to L^2(0,T;H) \times V
\] 
where
\[
	u \mapsto (u' + \A u + \lambda \B u, u(0)).
\]
We have $\Phi_\lambda = (1- \lambda) \Phi_0 + \lambda \Phi_1$, the mappings $\Phi_0$ and $\Phi_1$ are bounded by Proposition~\ref{prop:BV}
and by Theorem~\ref{thm:MRp1} the operator $\Phi_0$ is an isomorphism.
Now suppose the a priori-estimate
\[
	\norm{u}_{\MR_\fra(H)} \le c \norm{\Phi_\lambda(u)}_{L^2(0,T;H)\times V} \quad (u \in \MR_\fra(H), \lambda \in [0,1])
\]
holds for some $c > 0$.
Then by \cite[Theorem 5.2]{GT01} $\Phi_1$ is surjective.

Note that \eqref{eq:Bbound} and \eqref{eq:Bcoercive} hold with the same constants if we replace $\B$ by $\lambda \B$ where $0 \le \lambda \le 1$.
Hence the theorem is proved once we have established the following.
There exists a constant $C>0$ depending only on $c_H$, $T$, $\alpha$, $M$, $g(T)-g(0)$, $b$, $\delta$ and $\nu([0,T])$ such that
	\begin{equation*}
		\norm{u}_{\MR_\fra(H)}^2 \le C \left[\norm{u' + \A u + \B u}_{L^2(0,T;H)}^2 + \norm{u(0)}_V^2\right] \quad (u \in \MR_\fra(H)).
	\end{equation*}

	We may assume that $\fra = \fra^-$ and $g = g^-$.
	Let $u \in \MR_\fra(H)$. We set $f:= u' + (\A+\B)u$ and $u_0 := u(0)$.
	Then for $t \in [0,T]$ by Young's inequality for some $\varepsilon >0$ ($2st\le \varepsilon s^2 + \frac 1 \varepsilon t^2$, $s,t\in\R$), 
	\eqref{eq:Vbounded}, Proposition~\ref{prop:BV} \eqref{eq:BVformula}, 
	\eqref{eq:coercive} and \eqref{eq:Bcoercive},
	\begin{align*}
		&\tfrac 1 \varepsilon\norm{f}^2_{L^2(0,t;H)} + \varepsilon \norm{\A u}_{L^2(0,t;H)} + M \norm{u_0}_V^2\\
		& \ge 2 \int_0^t \Re(f \mid \A u)_H \ \d{s} + \fra^+(0,u_0,u_0)\\
		& = 2 \int_0^t \Re(u' \mid \A u)_H \ \d{s} + \fra^+(0,u_0,u_0) + 2 \int_0^t \Re((\A+\B)u \mid \A u)_H \ \d{s}\\
		& \ge \fra^-(t,u(t),u(t)) - \int_{(0,t)} \norm{u}_V^2 \ \d{\mu_g} + 2 \int_0^t \Re((\A+\B)u \mid \A u)_H \ \d{s}\\
		&\ge \alpha \norm{u(t)}^2_{V} - \int_{[0,t)} \norm{u}_V^2 \ \d{(\mu_g+\nu)} + 2\delta \int_0^t \norm{\A u}^2_H \ \d{s}.
	\end{align*}
	First we choose $\varepsilon = 2 \delta$, then
	\[
		\alpha \norm{u(t)}_V^2 \le \tfrac 1 {2 \delta} \norm{f}^2_{L^2(0,T;H)} + M \norm{u_0}_V^2 + \int_{[0,t)} \norm{u}_V^2 \ \d{(\mu_g+\nu)}
	\]
	for all $t \in [0,T]$. By Gronwall's inequality (see \cite[p.\ 498, Theorem~5.1]{EK86}) we obtain
	\[
		\norm{u(t)}_V^2 \le \tfrac 1 \alpha \left[ \tfrac 1 {2 \delta} \norm{f}^2_{L^2(0,T;H)} + M \norm{u_0}_V^2\right] \exp\left( \tfrac 1 \alpha (\mu_g+\nu)([0,T]) \right)
	\]
	for all $t \in [0,T]$. Thus
	\begin{equation}\label{eq:inftynorm}
		\norm{u}_{L^\infty(0,T;V)}^2 \le c_2 \left[ \tfrac 1 {2 \delta}\norm{f}^2_{L^2(0,T;H)} + M \norm{u_0}_V^2\right]
	\end{equation}
	where $c_2 := \frac 1 \alpha \exp\left( \tfrac 1 \alpha  (\mu_g+\nu)([0,T])  \right)$.
	
	Next we choose $\varepsilon= \delta$ and $t=T$, then
	\begin{multline*}
		\tfrac 1 \delta \norm{f}^2_{L^2(0,T;H)} + M \norm{u_0}_V^2 \ge\delta \norm{\A u}^2_{L^2(0,T;H)}
		  - (\mu_g+\nu)([0,T])\norm{u}_{L^\infty(0,T;V)}^2.
	\end{multline*}
	Thus by \eqref{eq:inftynorm} 
	\begin{equation}\label{eq:c_3est}
		c_3 \left[\tfrac 1 \delta \norm{f}^2_{L^2(0,T;H)} + M \norm{u_0}_V^2 \right] \ge\delta \norm{\A u}^2_{L^2(0,T;H)}.
	\end{equation}
	where $c_3 := 1+ c_2 (\mu_g+\nu)([0,T])$.
	
	Finally observe that $u' = f - (\A + \B) u$, thus by \eqref{eq:Bbound} and \eqref{eq:inftynorm}
	\begin{align*}
		\norm{u'}_{L^2(0,T;H)}
		& \le \norm{f}_{L^2(0,T;H)} + \norm{\A u}_{L^2(0,T;H)} +\norm{\B u}_{L^2(0,T;H)}\\
		&\le \norm{f}_{L^2(0,T;H)} + (1 + b) \norm{\A u}_{L^2(0,T;H)} + b \norm{u}_{L^\infty(0,T;V)}\\
		&\le \norm{f}_{L^2(0,T;H)} + (1 + b) \norm{\A u}_{L^2(0,T;H)} \\
		&\quad + b c_2^{1/2} \left[ \tfrac 1 {2 \delta}\norm{f}^2_{L^2(0,t;H)} + M \norm{u_0}_V^2\right]^{1/2}.
	\end{align*}
	This estimate together with \eqref{eq:c_3est} proves the claim.
\end{proof}

\begin{corollary}\label{cor:pert}
	Let $B\colon [0,T] \to \L(H)$ and $C \colon [0,T]\to \L (V,H)$ such that $B(.)v$ and $C(.)w$ are weakly measurable for all $v \in H$, $w \in V$.
	Suppose there exist constants $\beta_0 > 0$ and $\beta_1$ such that
	$\norm{B(t)}_{\L(H)} \le \beta_1$ for all $t \in [0,T]$ and
	\begin{equation*}
		(B(t) v \mid v)_H \ge \beta_0 \norm{v}_H^2 \quad (t\in [0,T], v \in H).
	\end{equation*}
	Moreover suppose there exists an integrable function $h \colon [0,T] \to [0,\infty)$ such that
	$\norm{C(t)}^2_{\L(V,H)}\le h(t)$ for all $t \in [0,T]$.
	Then for every $u_0\in V$, $f \in L^2(0,T;H)$ there exists a unique $u$ in $\MR_\fra(H)$ such that
	\begin{equation*}
		\left\{\begin{aligned}
			u' + B\A u + C u &= f\\
			u(0)&=u_0.
		\end{aligned}\right.
	\end{equation*}
		Moreover there exists a constant $C>0$ depending only on $c_H$, $T$, $\alpha$, $M$, $g(T)-g(0)$, $\beta_0$, $\beta_1$ and $\norm{h}_{L^1(0,T)}$ such that
	\begin{equation*} 
		\norm{u}_{\MR_\fra(H)}^2 \le C \left[\norm{f}_{L^2(0,T;H)}^2 + \norm{u_0}_V^2\right]. 
	\end{equation*}
\end{corollary}
\begin{proof}
	We define the operator $\B \colon W \to L^2(0,T;H)$ by $\B w := (B-1) \A w + C w$.
	The operators $\A$ and $\B$ satisfy the conditions of Theorem~\ref{thm:MR} and $\A w + \B w = B \A w + Cw$ in $L^2(0,T;H)$ for all $w \in W$.
\end{proof}

\section{An Aubin--Lions lemma for $\MR_\fra(H)$}

Let $V, H$ be Hilbert spaces over the field $\K$ such that $V \overset d \hookrightarrow H$.
\begin{theorem}\label{thm:A-L}
	Let
\[
	\fra \colon [0,T]\times V\times V \to \K
\]
be a $V$-bounded, coercive non-autonomous form and $\A \sim \fra$.
Suppose $V$ is compactly embedded in $H$.
Then $\MR_\fra(H)$ is compactly embedded in $L^2(0,T;V)$.
\end{theorem}
\begin{proof}
	Let $v \in \MR_\fra(H)$. Then
	\begin{multline}\label{eq:interpolationestimate}
		\alpha \norm{v}_{L^2(0,T;V)}^2 \le \int_0^T \Re \fra(t,v,v) \ \d{t} = \int_0^T \Re (\A v \mid v)_H \ \d{t}
			\\ \le \norm{\A v}_{L^2(0,T;H)} \norm{v}_{L^2(0,T;H)}.
	\end{multline}
	Moreover, if $c_H$ is the norm of the embedding of $V$ into $H$, then the above estimate implies
		\[
			\norm{v}_{L^2(0,T;V)} \le \frac {c_H}\alpha \norm{\A v}_{L^2(0,T;H)}.
		\]
	
	Let $(u_n)_{n\in\N}\subset \MR_\fra(H)$ be a bounded sequence.
	By the classical Aubin--Lions lemma (see \cite[Corollary~5]{Sim87}) we have that $L^2(0,T;V)\cap H^1(0,T;H)$ is compactly embedded in $L^2(0,T;H)$.
	Thus there exists a subsequence of $(u_n)_{n\in\N}$ which is Cauchy in $L^2(0,T;H)$.
	Finally by \eqref{eq:interpolationestimate} and the boundedness of $(u_n)_{n\in\N}$ in $\MR_\fra(H)$ we obtain that this subsequence is also Cauchy in $L^2(0,T;V)$.
	Hence $\MR_\fra(H)$ is compactly embedded in $L^2(0,T;V)$.
\end{proof}

\begin{corollary}\label{cor:A-L}
Let
\[
	\fra \colon [0,T]\times V\times V \to \K
\]
be a symmetric, $V$-bounded, coercive non-autonomous form of bounded variation and $\A \sim \fra$.
Suppose $V$ is compactly embedded in $H$.
Then $\MR_\fra(H)$ is compactly embedded in $L^p(0,T;V)$ for $1\le p<\infty$.
\end{corollary}
\begin{proof}
	Let $(u_n)_{n\in\N}\subset \MR_\fra(H)$ be a bounded sequence. 
	Then $(u_n)_{n\in\N}$ is also bounded in $C([0,T];V)$ and consequently in $L^p(0,T;V)$,
	since $\MR_\fra(H) \hookrightarrow C([0,T];V)$ by Theorem~\ref{thm:MRp1}.
	The sequence $(u_n)_{n\in\N}$ has a $L^2(0,T;V)$ convergent subsequence by Theorem~\ref{thm:A-L}. 
	This subsequence has a $t$-a.e.\ convergent subsequence.
	Hence by the dominated convergence theorem that this subsequence converges in $L^p(0,T;V)$.
\end{proof}

\section{Applications}
This section is devoted to some applications of the results given in the previous sections.
We give examples illustrating the theory without seeking for generality. 
The first two examples are similar to the examples given in \cite{ADLO14}. 
Here we have improved the condition on the time regularity.
The third example is inspired by \cite{AC10} and \cite{ADLO14}.
Compared to \cite{AC10} we consider a non-autonomous form and do not assume that the domain of the restriction of $\A(t)$ to $H$ is contained in $H^2_{loc}(\Omega)$,
but we assume that $\Omega$ is a abounded Lipschitz domain. It is possible to generalize our result to more general domains in a similar manner.
In relation to \cite{ADLO14} we weaken the condition on the time regularity of the non-autonomous form and allow a semilinear term.

Let $\Omega \subset \R^d$ be a bounded domain, where $d \in \N$.
In this section we always consider the Hilbert space $H:=L^2(\Omega)$ over the field $\K = \R$ or $\C$.

\subsection*{Elliptic operator with time dependent coefficients}
For simplicity we consider Dirichlet boundary conditions in this example.
Let $V=H^1_0(\Omega)$. 
Let $a_{jk} \colon [0,T] \times \Omega \to \K$ be measurable and bounded, where $j,k \in \{1,\dots d\}$.
Suppose $a_{jk}=\overline{a_{kj}}$ for all $j,k \in \{1,\dots d\}$ and there exists a constant $\alpha >0$, such that
\begin{equation*}
	\sum_{j,k=1}^d a_{jk} \xi_j \overline\xi_k \ge \alpha \abs{\xi}^2 \quad (\xi=(\xi_1, \dots, \xi_d) \in \K^d).
\end{equation*}
Moreover we suppose that there exists a bounded and non-decreasing function $g \colon [0,T] \to \R$ such that
\begin{equation*}
	\abs{a_{jk}(t,x)-a_{jk}(s,x)} \le g(t)-g(s) \quad (0\le t < s \le T, x \in \Omega)
\end{equation*}
for all $j,k \in \{1,\dots d\}$.
Let $b_j, c \in L^1(0,T;L^\infty(\Omega))$, where $j \in \{1,\dots d\}$.
Then we have the following.
\begin{proposition}
	For every $f \in L^2(0,T;H)$, $u_0 \in V$ there exists a unique $u \in C([0,T];V) \cap H^1(0,T;H)$ such that
	\begin{equation*}
		\left\{  \begin{aligned}
          		&u'(t)  -  \sum_{j,k=1}^d \partial_k (a_{jk}(t) \partial_j u(t) ) + b_j(t) \partial_j u(t) +c(t) u(t)=  f(t)\\
                       &u(0) =u_0.
		\end{aligned} \right.
	\end{equation*}
\end{proposition}
Note that the domain of the elliptic operator is time dependent.
\begin{proof}
	We define the non-autonomous form $\fra \colon [0,T] \times V \times V\to \K$ by 
	\begin{equation*}
		\fra(t,v,w) = \sum_{j,k=1}^d  a_{jk} \partial_j v \,\overline{\partial_k w} + (v \mid w)_H
	\end{equation*}
	and $C \colon [0,T]\to \L (V,H)$ by $Cv = \sum_{j=1}^d b_j(t) \partial_j v + (c-1) v$.
	Then $\fra$ and $C$ satisfy the conditions of Corollary~\ref{cor:pert}.
\end{proof}

\subsection*{Time dependent Robin boundary conditions}
Let $\Omega \subset \R^d$ be a  bounded domain  with Lipschitz boundary $\Gamma$. Let $V=H^1(\Omega)$.
Let  
\[
	\beta\colon [0,T]  \times \Gamma \to \K
\]
be a measurable function such that there exists a bounded and non-decreasing function $g \colon [0,T] \to \R$ such that
\begin{equation}\label{eq:RobinBV}
 	\lvert \beta(t,x) - \beta(s, x) \rvert \le g(t)-g(s) \quad (0\le t < s \le T, x \in \Gamma).
\end{equation}
\begin{proposition}
	For every $f \in L^2(0,T;H)$, $u_0 \in V$ there exists a unique $u \in C([0,T];V) \cap H^1(0,T;H)$ such that
	\begin{equation*}
\left\{  \begin{aligned}
           &u'(t)  - \Delta u(t)  = f(t)\\
           &u(0)=u_0\\
           &\partial_\nu u(t) + \beta(t,.) u = 0   \text{ on } \Gamma
                         \end{aligned} \right.
\end{equation*}
\end{proposition}
We denote by $\sigma$ the $(d-1)$-dimensional Hausdorff measure on $\Gamma$
and define the normal derivative in the following way.
Let $v \in V$ such that $\Delta v \in H$ and let $h \in L^2(\Gamma, \d \sigma)$.  
Then $\partial_\nu v = h$ by definition if
$\int_\Omega \nabla v \nabla w + \int_\Omega \Delta v w = \int_\Gamma h \T w \, \d \sigma$ for all $w \in V$,
where $\T\colon V \to L^2(\Gamma, \d \sigma)$ denotes the trace operator.

Note that we could replace the Laplacian with an elliptic operator as in the previous example.
\begin{proof}
We consider the symmetric form
\[
	\fra\colon [0,T] \times V \times V \to \K
\]
defined by 
\begin{equation*}
	\fra(t, v, w) = \int_\Omega \nabla v \nabla\overline{w}\ \d x + \int_\Gamma \beta(t, .) \T v \overline{\T w}\ \d\sigma.
\end{equation*}
Where $\operatorname{T}\colon V \to L^2(\Gamma, \d \sigma)$ denotes the trace operator and $\sigma$ the $(d-1)$-dimensional Hausdorff measure on $\Gamma$.
The form $\fra$ is symmetric and $V$-bounded.
Moreover $\fra$ is quasi-coercive; i.e., $\fra +\lambda (.\mid .)_H$ is coercive for some $\lambda > 0$.
This is a consequence of the inequality 
\begin{equation}\label{trace-comp}
\int_\Gamma \lvert u \rvert^2 \ \d\sigma \le \epsilon \norm u_{H^1}^2 + c_\epsilon \norm u_{L^2(\Omega)}^2,
\end{equation}
which is valid for all $\epsilon > 0$ ($c_\epsilon$ is a constant depending on $\epsilon$). 
Finally $\fra$ is of bounded variation by \eqref{eq:RobinBV}.
Now the proposition follows by Corollary~\ref{cor:pert}.
\end{proof}

\subsection*{Existence of a quasi-linear problem}
For this example we consider $\Omega$, $H$, $V$ and $\fra$ from one of the previous examples. 
We set $\A \sim \fra$.
Note that $H^1_0(\Omega)$ is compactly embedded in $L^2(\Omega)$ and
$H^1(\Omega)$ is compactly embedded in $L^2(\Omega)$ if $\Omega$ is a Lipschitz domain by Rellich's theorem.

Let $0< \beta_0 < \beta_1$ and let 
\[
	m \colon [0,T]\times \Omega \times \R^{d+1} \to [\beta_0,\beta_1]
\] 
and suppose that
$m(t,x,.) \colon \R^{d+1} \to [\beta_0,\beta_1]$ is continuous for a.e.\ $(t, x) \in [0,T] \times \Omega$.

Moreover let 
\[
	f  \colon [0,T]\times \Omega \times \R^{d+1} \to \K
\] 
be measurable.
Suppose that $f(t,x,.) \colon \R^{d+1} \to \K$ is continuous for a.e.\ $(t, x) \in [0,T] \times \Omega$ and
\[
	\abs{f(t,x,v)}^2 \le g(t,x)^2 + h(t) \abs{v}^2 \quad (\text{a.e.\ }(t, x) \in [0,T] \times \Omega, v \in \R^{d+1})
\]
for some $g \in L^2(0,T;H)$ and some non-negative $h \in L^q(0,T)$, where $q >1$.

\begin{proposition}\label{prop:NLP}
For every $u_0 \in V$ there exists an $u \in \MR_\fra(H)$ such that
	\begin{equation}\label{eq:NLCP}
		\left\{  \begin{aligned}
          		&u'  + m(t,x,u,\nabla u) \A u=  f(t,x,u,\nabla u)\\
                       &u(0) =u_0.
		\end{aligned} \right.
	\end{equation}
\end{proposition}
For the proof we need Schaefer's fixed point theorem (see \cite{Eva98}[p.\ 504]).
	\begin{lemma}[Schaefer's fixed point theorem]\label{lem:Schaefer}
		Let $X$ be a Banach space. 
		Suppose $S \colon X \to X$ is a continuous and compact mapping. Assume further that
		the set $\{ u \in X : u=\lambda S u, \lambda \in [0,1] \}$ is bounded. Then $S$ has a fixed point.
	\end{lemma}
\begin{proof}[Proof of Proposition~\ref{prop:NLP}]
	Let $u_0 \in V$ and $\frac 2 p =1-\frac 1 q$.
	Note that by Rellich's theorem $V$ is compactly embedded in $H$.
	Thus we obtain by Corollary~\ref{cor:A-L} that $\MR_\fra(H)$ is compactly embedded in $L^p(0,T;V)$.
	
	For $v \in L^p(0,T;V)$ we have $f(v) \in L^2(0,T;H)$. 
	Thus by Corollary~\ref{cor:pert} there exists a unique $u_v \in \MR_\fra(H)$ such that
	\begin{equation*}
		\left\{  \begin{aligned}
          		&u_v'  + m(t,x,v,\nabla v) \A u_v=  f(t,x,v,\nabla v)\\
                       &u_v(0) =u_0.
		\end{aligned} \right.
	\end{equation*}
	We define the mapping
	\[
		S \colon L^p(0,T;V) \to L^p(0,T;V)
	\]
	by $v \mapsto u_v$.
	Now if $u$ is a fixed point of $S$, then $u$ is a solution of \eqref{eq:NLCP}.
	We show that $S$ satisfies the conditions of Lemma~\ref{lem:Schaefer}.
	
	First we show that $S$ is continuous. Let $v_n \to v$ in $L^p(0,T;V)$.
	In order to show that $S$ is continuous, it suffices to show that there exists a subsequence
	of $S v_n$ which converges to $Sv$ in $L^p(0,T;V)$.
	By extracting a subsequence we may assume that $(v_n(t,x), \nabla v_n(t,x))$ converges to $((v(t,x),\nabla v(t,x))$ for a.e.\ $(t,x) \in [0,T]\times \Omega$.
	and there exists a $\tilde v \in L^p(0,T; L^2(\Omega)^{d+1}))$ such that $\abs{(v_n(t,x),\nabla v_n(t,x))} \le \abs{\tilde v(t,x)}$ for a.e.\ $(t,x) \in [0,T]\times \Omega$ and all $n \in \N$.
	Thus 
	\[
		\norm{f(v_n)}_{L^2(0,T;H)}^2 \le \norm{g}_{L^2(0,T;H)}^2 + \norm{h}_{L^q(0,T)} \norm{\tilde v}_{L^p(0,T;L^2(\Omega)^{d+1})}
	\]
	for all $n \in \N$.
	We obtain by Corollary~\ref{cor:pert} that $(Sv_n)_{n \in \N}$ is bounded in $\MR_\fra(H)$.
	By extracting another subsequence we may assume that $Sv_n \rightharpoonup u$ in $\MR_\fra(H)$ for some $u \in \MR_\fra(H)$.
	Moreover by the compactness of the embedding of $\MR_\fra(H)$ in $L^p(0,T;V)$ we also may assume (by extracting another subsequence)
	that $Sv_n \to u$ in $L^p(0,T;V)$.
	It remains to show that $u = Sv$ or equivalently that $u$ solves
	\begin{equation*}
		\left\{  \begin{aligned}
          		&u'  + m(t,x,v,\nabla v) \A u=  f(t,x,v,\nabla v)\\
                       &u(0) =u_0.
		\end{aligned} \right.
	\end{equation*}
	We have by the continuity assumptions on $m$ and $f$ that
	\[
		m(t,x,v_n(t,x),\nabla v_n(t,x)) \to m(t,x,v(t,x),\nabla v(t,x)) \quad (n \to \infty)
	\]
	and
	\[
		f(t,x,v_n(t,x),\nabla v_n(t,x)) \to f(t,x,v(t,x),\nabla v(t,x)) \quad (n \to \infty)
	\]
	for a.e.\ $(t,x) \in [0,T]\times \Omega$.
	Hence by the dominated convergence theorem that
	\[
		m(t,x,v_n,\nabla v_n)w \to m(t,x,v,\nabla v)w \quad (n \to \infty) \text{ in }L^2(0,T;H) 
	\]
	for any $w \in L^2(0,T;H)$ and
	\[
		f(t,x,v_n,\nabla v_n) \to f(t,x,v,\nabla v) \quad (n \to \infty) \text{ in }L^2(0,T;H).
	\]
	We set $u_n:= Sv_n$.
	For $w \in L^2(0,T;H)$ we obtain
	\begin{align*}
		(f(v) \mid w)_{L^2(0,T;H)} &= \lim_{n\to \infty} (f(v_n) \mid w)_{L^2(0,T;H)}\\
			&= \lim_{n\to \infty} (u_n'+m(t,x,v_n,\nabla v_n)\A u_n \mid w)_{L^2(0,T;H)}\\
			&= \lim_{n\to \infty} (u_n'\mid w)_{L^2(0,T;H)} + (\A u_n \mid  m(t,x,v_n,\nabla v_n) w)_{L^2(0,T;H)}\\
			&= (u'\mid w)_{L^2(0,T;H)} + (\A u \mid  m(t,x,v,\nabla v) w)_{L^2(0,T;H)}\\
			&= (u' + m(t,x,v,\nabla v)\A u \mid   w)_{L^2(0,T;H)}.
	\end{align*}
	Finally, the weak convergence of $u_n$ to $u$ in $\MR_\fra(H)$ together with the embedding $\MR_\fra(H) \hookrightarrow C([0,T];V)$ shows that $u_n(0) \rightharpoonup u(0)$ in $V$.
	Hence $u(0)=u_0$ and thus $u = S v$.
	
	The mapping $S$ is compact, since the image of $S$ is contained in $\MR_\fra(H)$ and $\MR_\fra(H)$ is compactly embedded in $L^p(0,T;V)$.
	
	It remains to show that the set $\{ u \in L^p(0,T;H) : u=\lambda S u, \lambda \in [0,1] \}$ is bounded.
	Let $\lambda \in (0,1]$ and $u \in L^p(0,T;H)$ such that $u=\lambda S u$.
	Note that $u \in \MR_\fra(H)$.
	We have $u' +m(t,x,u,\nabla u) \A u = \lambda f(u)$ and $u(0)=u_0$.
	Let $t \in [0,T]$. We obtain by Proposition~\ref{prop:BV}
	\begin{align*}
		&\tfrac 1 \beta_0 \norm{f(u)}^2_{L^2(0,t;H)} + \beta_0 \norm{\A u}_{L^2(0,t;H)} + M \norm{u_0}_V^2\\
		& \ge 2 \int_0^t \Re(f(u) \mid \A u)_H \ \d{s} + \fra^+(0,u_0,u_0)\\
		& = 2 \int_0^t \Re(u' \mid \A u)_H \ \d{s} + \fra^+(0,u_0,u_0) + 2 \int_0^t \Re(m(t,x,u,\nabla u) \A u \mid \A u)_H \ \d{s}\\
		& \ge \fra^-(t,u(t),u(t)) - \int_{(0,t)} \norm{u}_V^2 \ \d{\mu_g} + 2 \beta_0\int_0^t \norm{\A u}^2_H \ \d{s}\\
		&\ge \alpha \norm{u(t)}^2_{V} - \int_{(0,t)} \norm{u}_V^2 \ \d{\mu_g} + 2 \beta_0\int_0^t \norm{\A u}^2_H \ \d{s}.
	\end{align*}
	By our assumption on $f$ we obtain the estimate
	\[
		\norm{f(u)}_{L^2(0,t;H)}^2 \le \norm{g}_{L^2(0,t;H)}^2 + \norm{h}_{L^q(0,t)} \norm{u}_{L^p(0,T;V)}.
	\]
	A combination of the above estimates yields
	\[
		\alpha \norm{u(t)}^2_{V} + \beta_0 \norm{\A u}_{L^2(0,t;H)} \le M \norm{u_0}_V^2 + \norm{g}_{L^2(0,T;H)}^2 + \int_{(0,t)} \norm{u}_V^2 \ (\d{\mu_g } + h\ \d s).
	\]
	Finally an application of Gronwall's lemma proves the claim.
\end{proof}
Note that the set $\{ u \in L^p(0,T;H) : u=\lambda S u, \lambda \in [0,1] \}$ is even bounded in $\MR_\fra(H)$.

\bibliographystyle{amsalpha}
\bibliography{Bibliography}

\noindent
\emph{Dominik Dier}, Institute of Applied Analysis, 
University of Ulm, 89069 Ulm, Germany,
\texttt{dominik.dier@uni-ulm.de}

\end{document}